\documentclass{article}

\usepackage{amsmath,amssymb,amsthm,bm}
\usepackage{bbm}
\usepackage{enumerate}
\usepackage{pifont}
\usepackage{setspace}
\usepackage{tikz-cd} 
\usepackage[utf8]{inputenc}
\usepackage{tikz}
\usepackage{graphicx}
\usepackage{pgfplots}
\pgfplotsset{compat=1.14}

\newcommand{\ZB}{\mathbb{Z}}
\newcommand{\DB}{\mathbb{D}}

\newcommand{\RB}{\mathbb{R}} 
 
\newcommand{\CB}{\mathbb{C}}
\newcommand{\1}{\mathbbm{1}}

\newcommand{\FC}{\mathcal{F}}

\newcommand{\LC}{\mathcal{L}}
\newcommand{\MC}{\mathcal{M}}

\newcommand{\PC}{\mathcal{P}}

\newcommand{\RC}{\mathcal{R}}
\newcommand{\SC}{\mathcal{S}}

\newcommand*\conj[1]{\overline{#1}}

\newcommand*\supp{\mathrm{supp}}

\renewcommand{\Re}{\operatorname{Re}}
\renewcommand{\Im}{\operatorname{Im}}

\theoremstyle{plain}
\newtheorem{theorem}{Theorem}[section]
\newtheorem{lemma}[theorem]{Lemma}

\theoremstyle{definition}

\newtheorem{example}[theorem]{Example}

\theoremstyle{remark}
\newtheorem{remark}[theorem]{Remark}

\title{On Laplace--Carleson embeddings, and $L^p$-mapping properties of the Fourier transform}
\author{Eskil Rydhe\thanks{eskil.rydhe@math.lu.se, Centre for Mathematical Sciences, Lund University, Sweden. This work was supported by the Knut and Alice Wallenberg foundation, scholarship KAW 2016.0442, and produced while the author was a postdoc at University of Leeds, UK.}}

\begin{document}

\maketitle
	
\begin{abstract}
We investigate so-called Laplace--Carleson embeddings for large exponents. In particular, we extend some results by Jacob, Partington, and Pott. We also discuss some related results for Sobolev- and Besov spaces, and mapping properties of the Fourier transform. These variants of the Hausdorff--Young theorem appear difficult to find in the literature. We conclude the paper with an example related to an open problem.
\end{abstract}

\section{Introduction}\label{sec:Introduction}
Throughout this note we let $1\le p,q\le \infty$, and $p'=\frac{p}{p-1}$, so that $\frac{1}{p}+\frac{1}{p'}=1$. By $\RB_+$ and $\CB_+$ we respectively denote the set of positive real numbers $(0,\infty)$ and the complex upper half plane $\left\{z\in\CB \mid \Im z>0\right\}$. We let $\mu$ be a positive Borel measure on $\CB_+$. Preliminaries and notation not covered in this section is deferred to Section~\ref{sec:PreliminariesAndNotatiion}. In particular, we postpone the definitions of the following standard function spaces.

\medskip
\begin{tabular}{lcl}
	$H^p(\CB_+)$ & - & Hardy space of analytic functions on $\CB_+$;\\[2pt]
	$A^p_\alpha(\CB_+)$ & - & Standard weighted Bergman space;\\[2pt]
	$W^p_s$ & - & Sobolev space of tempered distributions on $\RB^d$;\\[2pt]
	$F^{p,q}_s$ & - & Triebel--Lizorkin space on $\RB^d$;\\[2pt]
	$B^{p,q}_s$ & - & Besov space on $\RB^d$;\\[2pt]
	$\dot X$ & - & Homogeneous counterpart of $X\in\{W^p_s,F^{p,q}_s,B^{p,q}_s\}$.
\end{tabular}
\medskip

The notion of \textit{Laplace--Carleson embeddings} was coined in \cite{Jacob-Partington-Pott2013:OnLaplace-CarlesonEmbeddingTheorems}, and refers to maps of the type
\begin{align*}
\LC\colon L^p(\RB_+)\to L^q(\CB_+,d \mu),\quad f\mapsto \LC f:=\int_0^\infty f(t)e^{2\pi i t \cdot}\,d  t.
\end{align*}
A priori, the above map is strictly formal. However, if $\LC L^p(\RB_+)$ is indeed contained in $L^q(\CB_+,d  \mu)$, then the inclusion is continuous by the closed graph theorem. 

One may of course also consider more general spaces in place of $L^p(\RB_+)$, for example weighted Sobolev spaces. In this direction, we will only summarily consider spaces of order $0$ and with simple weights.

We now recall some basic problems and results from \cite{Jacob-Partington-Pott2013:OnLaplace-CarlesonEmbeddingTheorems}: Given an interval $I\subset \RB$, with length $|I|$, we define the so-called Carleson box
\begin{align*}
Q_I:=\left\{x+iy\mid x\in I,0<y\le |I|\right\}\subset \CB_+.
\end{align*}
If $1\le p,q<\infty$, and $\LC\colon L^p(\RB_+)\to L^q(\CB_+,d  \mu)$ is bounded, then the measure $\mu$ necessarily satisfies
\begin{align}\label{eq:CarlesonCondition}
\mu(Q_I)\lesssim|I|^{q/p'}\quad\textnormal{for all intervals}\quad I\subset\RB.
\end{align}
The motivation for this paper arose from the question of to which extent the necessary condition \eqref{eq:CarlesonCondition} is also sufficient for $\LC\colon L^p(\RB_+)\to L^q(\CB_+,d  \mu)$ to be bounded. The following results can be found in \cite[Section~3]{Jacob-Partington-Pott2013:OnLaplace-CarlesonEmbeddingTheorems}.
\begin{enumerate}[(I)]
	\item If $1\le p\le 2$ and $p'\le q<\infty$, then \eqref{eq:CarlesonCondition} is also sufficient for $\LC\colon L^p(\RB_+)\to L^q(\CB_+,d  \mu)$ to be bounded.\label{page:CaseI}
	\item If $\mu$ is sectorial, i.e. there exists a $c>0$ such that $\mu$ has support in the sector $\left\{z\in\CB_+ \mid   \Im z\ge c |\Re z|\right\}$, and $2<p\le q<\infty$, then \eqref{eq:CarlesonCondition} is sufficient.
	\item If $\mu$ is sectorial, $1< p \le 2$, and $p\le q<\infty$, then \eqref{eq:CarlesonCondition} is sufficient.
	\item If $1\le q<p<\infty$, then \eqref{eq:CarlesonCondition} is not sufficient, even under the assumption that $\mu$ has support on the imaginary axis.
\end{enumerate}

It may be useful for orientation to consult the $\left(1/p,1/q\right)$-diagram in Figure~\ref{figure:KnownAndNewResults}. Our primary contribution to this body of knowledge is that the hypothesis of sectoriality may be removed in case (II).
\begin{theorem}\label{theorem:LaplaceCarlesonEmbeddingTheorem}
	If $2< p\le q<\infty$, and \eqref{eq:CarlesonCondition} holds, then $\LC\colon L^p(\RB_+)\to L^q(\CB_+,d \mu)$ is bounded.
\end{theorem}

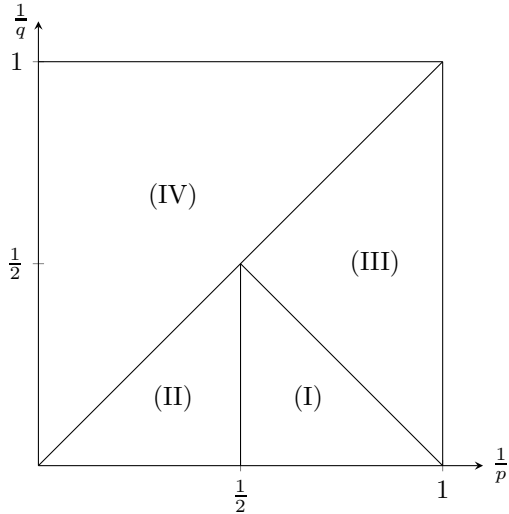
\begin{figure}[h!]
	\begin{center}
		\begin{tikzpicture}
		\begin{axis}[xmin=0, xmax=1.1, ymin=0, ymax=1.1, 
		axis lines = left,
		x label style={at={(axis description cs:1,0)},anchor=west},
		y label style={at={(axis description cs:0,1)},rotate=270,anchor=east},
		xlabel={$\frac{1}{p}$},
		ylabel={$\frac{1}{q}$},
		scaled x ticks=false,
		scaled y ticks=false,
		xtick={0.5,1},
		xticklabels={$\frac{1}{2}$,$1$},
		ytick={0.5,1},
		yticklabels={$\frac{1}{2}$,$1$},
		width=0.618\textwidth,
		height=0.618\textwidth
		]
		\addplot[black] coordinates
		{(0,0) (1,1)};
		\addplot[black] coordinates
		{(0.5,0.5) (1,0)};
		\addplot[black] coordinates
		{(0,1) (1,1)};
		\addplot[black] coordinates
		{(1,1) (1,0)};
		\addplot[black] coordinates
		{(0.5,0.5) (.5,0)};
		\node[black] at (axis cs:0.667,0.167){(I)};
		\node[black] at (axis cs:0.333,0.167){(II)};
		\node[black] at (axis cs:0.833,0.5){(III)};
		\node[black] at (axis cs:0.333,0.667){(IV)};
		\end{axis}
		\end{tikzpicture}
	\end{center}
	\caption{It is previously known that for $(p,q)$ corresponding to the region labelled (I), condition \eqref{eq:CarlesonCondition} is necessary and sufficient for the Laplace--Carleson embedding $\LC\colon  L^p(\RB_+)\to L^q(\CB_+,\mathrm{d}{\mu})$ to be bounded. In the regions (II) and (III), \eqref{eq:CarlesonCondition} is necessary and sufficient under the additional hypothesis that $\mu$ is sectorial. In (IV), \eqref{eq:CarlesonCondition} is not sufficient, even for sectorial measures. Theorem~\ref{theorem:LaplaceCarlesonEmbeddingTheorem} states that in (II), \eqref{eq:CarlesonCondition} is necessary and sufficient without any particular conditions on the measure.}\label{figure:KnownAndNewResults}
\end{figure}

Consider the case (I), i.e. $1\le p\le 2$, $p'\le q<\infty$. The proof that \eqref{eq:CarlesonCondition} is sufficient in this case consists of two main steps: The Hausdorff--Young theorem readily implies that $\LC$ is a bounded map from $L^p(\RB_+)$ to $H^{p'}(\CB_+)$, the standard Hardy space of the upper half plane. The Carleson--Duren embedding theorem (Theorem~\ref{theorem:CarlesonDuren} below) then states that $H^{p'}(\CB_+)\hookrightarrow L^q(\CB_+,d \mu)$ if and only if $\mu$ satisfies \eqref{eq:CarlesonCondition}. The proof of Theorem~\ref{theorem:LaplaceCarlesonEmbeddingTheorem} has the same structure:

We let $A^p_\alpha(\CB_+)$ denote the standard weighted Bergman space of analytic functions on $\CB_+$. For $p>2$, we have the following substitute for the Hausdorff--Young theorem:
\begin{theorem}\label{theorem:LaplaceBergmanEmbedding}
	If $2<p\le q<\infty$, then $\LC\colon L^p(\RB_+)\to A^q_{q/p'
		-2}(\CB_+)$ is bounded.
\end{theorem}
\begin{remark}
	By case (I), Theorem~\ref{theorem:LaplaceBergmanEmbedding} remains valid for $p=2$, provided that $q>2$.
\end{remark}

Theorem~\ref{theorem:LaplaceCarlesonEmbeddingTheorem} is immediate from Theorem~\ref{theorem:LaplaceBergmanEmbedding}, and a Carleson embedding type theorem for Bergman spaces, stated below as Theorem~\ref{theorem:CarlesonBergman}.

For readers with a particular interest in Bergman spaces, we also derive an analogue for analytic functions on the open unit disk $\DB$. We let $dA$ signify integration with respect to area measure on $\CB$.
\begin{theorem}\label{theorem:ZBergmanEmbedding}
	If $2<p\le q<\infty$, then there exists $C=C_{p,q}>0$ such that
	\[
	\left(\int_\DB \left|\sum_{k=0}^\infty a_kw^k\right|^q(1-|w|^2)^{q/p'-2}\,d  A(w)\right)^{1/q}\le C\left(\sum_{k=0}^\infty |a_k|^p\right)^{1/p}
	\]
	for any sequence $(a_k)_{k=0}^\infty$.
\end{theorem}

We also obtain some results for the power weighted spaces $L^p(\RB_+,x^\alpha\,d  x)$. The next result is a simultaneous analogue of Theorem~\ref{theorem:LaplaceBergmanEmbedding} and \cite[Theorem~1]{Duren-Gallardo--Gutierrez-Montes--Rodriguez2007:APaley--WienerTheoremForBergmanSpacesWithApplicationToInvariantSubspaces}.
\begin{theorem}\label{theorem:WeightedLaplaceBergmanEmbedding}
	If $2<p\le q<\infty$, and $\alpha <p/q'-1$, then 
	\[
	\LC\colon L^p(\RB_+,x^\alpha\,d  x)\to A^q_{q/p'-2-\alpha q/p}(\CB_+)
	\]
	is bounded.
\end{theorem}
We also obtain:
\begin{theorem}\label{theorem:WeightedLaplaceHardyEmbedding}
	If $2<p<\infty$, then 
	\[
	\LC\colon L^p(\RB_+,x^{p-2}\,d  x)\to H^p(\CB_+)
	\]
	is bounded.
\end{theorem}
A weighted analogue of Theorem~\ref{theorem:LaplaceCarlesonEmbeddingTheorem} becomes:
\begin{theorem}\label{theorem:WeightedLaplaceCarlesonEmbedding}
	Let $2<p\le q<\infty$, and $\alpha \le p/q'-1$. Then 
	\[
	\LC\colon L^p(\RB_+,x^\alpha\,d  x)\to L^q(\CB_+,d \mu)
	\]
	is bounded if and only if $\mu$ satisfies
	\[
	\mu(Q_I)\lesssim|I|^{q/p'-\alpha q/p}\quad\textnormal{for all intervals}\quad I\subset\RB.
	\]
\end{theorem}

We now transition into a discussion about the Fourier transform $\FC$, and the Hausdorff--Young theorem. In what follows, the underlying domain of any space of distributions is $\RB^d$, unless we indicate otherwise. For example, $L^p$ denotes $L^p(\RB^d)$. We let $|x|$ denote the Euclidean norm of $x=(x_1,\ldots,x_d)\in\RB^d$.

The Hausdorff--Young theorem states that if $1\le p\le 2$, then $\FC\colon L^p\to L^{p'}$, or equivalently $\FC^{-1}\colon L^p\to L^{p'}$. The original version of this result was an analogous statement about periodic functions, see \cite{Young1913:OnTheDeterminationOfTheSummabilityOfAFunctionByMeansOfItsFourierConstants} and \cite{Hausdorff1923:EineAusdehnungDesParsevalschenSatzesUberFourierreihen}, whereas the essence of the present statement is found in \cite{Titchmarsh1924:AContributionToTheTheoryOfFourierTransforms}. For a more careful historical account, we refer to the survey \cite{Butzer1994:TheHausdorff-YoungTheoremsOfFourierAnalysisAndTheirImpact}.

If $p>2$ and $f\in L^p$, then $\check f:=\FC^{-1}f$ in general needs to be interpreted as a tempered distribution. As an indication of this, we mention a theorem by Hardy and Littlewood \cite[p.~237]{Hardy-Littlewood1914:SomeProblemsOfDiophantineApproximation}, stating that the formal series $\sum_{k=1}^{\infty}\frac{1}{k^{1/2}}\cos (k^2\pi x)$ is not the Fourier series of any function. 


Interpreting $\check f$ as the distributional boundary values of $\LC f$, Theorem~\ref{theorem:LaplaceBergmanEmbedding} gives us a quantitative estimate on the regularity of $\check f$. The proof of Theorem~\ref{theorem:LaplaceBergmanEmbedding} is based on the relation between $\LC$ and $\FC^{-1}$, iterated use of the Plancherel theorem, and complex interpolation. By a similar (in fact simpler) argument we obtain a stronger result:
\begin{theorem}\label{theorem:HardyLittlewoodInequalityNonRadial}
	Let $p\ge 2$. If $f\in L^1$, then
	\begin{align*}
	\int_{\RB^d}|\hat f(\xi)|^p\,d  \xi
	\lesssim
	\int_{\RB^d}|f(x)|^{p}\left(\prod_{k=1}^d|x_k|\right)^{p-2}\,d  x.
	\end{align*}
\end{theorem}

While Theorem~\ref{theorem:LaplaceBergmanEmbedding}, and the proof leading up to Theorem~\ref{theorem:HardyLittlewoodInequalityNonRadial}, was discovered independently, the corresponding theorem for periodic functions of one variable dates back to Hardy and Littlewood \cite[Theorem~3]{Hardy-Littlewood1927:SomeNewPropertiesOfFourierConstants}. By the inequality of geometric and arithmetic means, and the equivalence of norms on $\RB^d$, Theorem~\ref{theorem:HardyLittlewoodInequalityNonRadial} implies the following result, which appears to be a folklore generalization of the theorem by Hardy and Littlewood.
\begin{theorem}\label{theorem:HardyLittlewoodInequalityRadial}
	Let $p\ge 2$. If $f\in L^1$, then
	\begin{align*}
	\int_{\RB^d}|\hat f(\xi)|^p\,d  \xi
	\lesssim
	\int_{\RB^d}|f(x)|^{p}|x|^{(p-2)d}\,d  x.
	\end{align*}
\end{theorem}
Even though Theorem~\ref{theorem:LaplaceBergmanEmbedding} will eventually be derived from Theorem~\ref{theorem:HardyLittlewoodInequalityRadial}, Theorem~\ref{theorem:HardyLittlewoodInequalityNonRadial} seems interesting in its own right, as an example of a weighted inequality for the Fourier transform, where the weight is non-radial.

Let $W^{p}_s$ denote the standard Sobolev space of fractional order $s$, and $\dot W^{p}_s$ its homogeneous counterpart. Although Theorem~\ref{theorem:HardyLittlewoodInequalityRadial} is part of the folklore, the following (nearly immediate) consequence appears to be absent in the literature.
\begin{theorem}\label{theorem:LebesgueToHomogneneousSobolev}
	If $p\ge 2$, then $\FC\colon L^p\to\dot W^{p}_{d(2/p-1)}$ is bounded.
\end{theorem}

A technical remark may be in order. Take $p>2$, and $s=d(2/p-1)$. In particular, $s<0$. Moreover, let $f\in L^p$. If in addition, $|f|^p|\cdot|^{(p-2)d}$ is integrable, then the Riesz potential fractional order antiderivative $\dot I_{s}\hat f$ is well-defined as a tempered distribution, and $\dot I_{s}\hat f\in L^p$. For general $f\in L^p$, this understanding of $\dot I_{s}\hat f$ is to naive. Instead, one needs to identify $\hat f\in\SC'$ with the equivalence class $[\hat f]\in\SC'/\PC$, where $\PC$ denotes the space of polynomials. Then $\dot I_{s}[\hat f]\in \SC'/\PC$ is well defined. The Littlewood--Paley theorem offers a canonical way to identify $\dot I_{s}[\hat f]$ with an element of $L^p$. By said identification, $\dot W^{p}_s$ becomes a proper subspace of $W^{p}_s$ (recall that $s<0$). We therefore obtain a variation of the above result:

\begin{theorem}\label{theorem:LebesgueToNonHomogneneousSobolev}
	If $p\ge 2$, then $\FC\colon L^p\to W^{p}_{d(2/p-1)}$ is bounded.
\end{theorem}

This result is less subtle: If $f\in L^p$, then the Bessel potential fractional antiderivative $I_s\hat f$ is a tempered distribution, and an element of $L^p$.

By the duality $\left(\dot W^p_s\right)'=\dot W^{p'}_{-s}$ we obtain:
\begin{theorem}\label{theorem:LebesgueToNonHomogneneousSobolevDual}
	If $p\in(1,2)$, then $\FC\colon \dot W^{p}_{d(2/p-1)}\to L^p$ is bounded.
\end{theorem}
A related observation is that Theorem~\ref{theorem:LebesgueToHomogneneousSobolev} does not extend to $p<2$. Indeed, if $\FC\colon L^p\to \dot W^{p}_{d(2/p-1)}$ was bounded for some $p\in (1,2)$, then $\FC\colon \dot W^{p'}_{d(2/p'-1)}\to L^{p'}$ would also be bounded, again by duality. Since all the function spaces in question are invariant under the reflection operator $\RC=\FC^2$, we would have that $\FC^{-1}=\FC^3\colon \dot W^{p'}_{d(2/p'-1)}\to L^{p'}$. By Theorem~\ref{theorem:LebesgueToHomogneneousSobolev}, this map would be invertible, and $\FC\colon L^{p'}\to \dot W^{p'}_{d(2/p'-1)}$ would be bounded below. But this is not true. Consider for example the (essentially) $L^{p'}$-normalized indicator function $R^{(d-1)/p'}\1_{A_R}$ of the annulus $A_R=\left\{x\in\RB^d;  R<|x|<R+1\right\}$. It is easy to show that $\|\1_{A_R}\|_{L^{p'}}\approx 1$, while $\|\hat \1_{A_R}\|_{\dot W^{p'}_{d(2/p'-1)}}\to 0$ as $R\to\infty$.

Using the formalism of homogeneous Triebel--Lizorkin-spaces, we note that $\dot W^{p}_s=\dot F^{p,2}_s\subset \dot F^{p,p}_{s}$ for $p>2$, since the spaces $\dot F^{p,q}_s$ increase with $q$. A more general statement is that $\dot W^{p}_s\subset \dot F^{q,q}_{s_q}$, provided that $q\ge p>2$, and $s-d/p=s_q-d/q$. This follows from a standard embedding result, stated below as Theorem~\ref{theorem:TriebelLizorkinEmbeddingTheorem}. Theorem~\ref{theorem:LebesgueToHomogneneousSobolev} implies that $\FC^{-1}=\FC^3\colon L^p\to\dot F^{q,q}_{d(1/q-1/p')}$. Theorem~\ref{theorem:LaplaceBergmanEmbedding} is now a consequence of the relation between $\FC^{-1}$ and $\LC$, and the fact that the analytic part of $\dot F^{q,q}_{s}(\RB)$ is contained in $A^{q}_{-sq-1}(\CB_+)$ when $s<0$.

We briefly compare Theorem~\ref{theorem:LebesgueToNonHomogneneousSobolev} with a result by Hörmander \cite[Theorem~7.9.3]{Hormander1990:TheAnalysisOfLinearPartialDifferentialOperatorsI}: If $p>2$ and $s<d(1/p-1/2)$, then $\FC\colon L^p\to W^2_s$. By Theorem~\ref{theorem:TriebelLizorkinEmbeddingTheorem}, this implies that $\FC\colon L^p\to W^p_s$ whenever $s<d(2/p-1)$. In relation to this, we point out that the target space in Theorem~\ref{theorem:LebesgueToHomogneneousSobolev} is optimal within the scale of homogeneous Sobolev spaces, and at least close to optimal in terms of Triebel--Lizorkin spaces. The next result is a precise formulation of this statement.
\begin{theorem}\label{theorem:OptimalityOfResults}
	Let $2<p<\infty$, $1< r,q<\infty$, and $s\in\RB$. If $\FC\colon L^p\to \dot F^{r,q}_s$ is bounded, then $s=d(1/p-1/r')$. Moreover, it holds that $r\ge p$, and if $r>p$, then $\dot W^p_{d(2/p-1)}\subsetneq \dot F^{r,q}_s$. In particular, if $\FC\colon L^p\to \dot W^r_s$ is bounded, then $\dot W^p_{d(2/p-1)}\subseteq\dot W^r_s$, with equality if and only if $r=p$ and $s=d(2/p-1)$.
\end{theorem}

Having discussed $\FC L^p$ for $p>2$, it seems natural to add an observation about $p<2$: The typical proof of the Hausdorff--Young inequality uses complex (Riesz--Thorin) interpolation between $\FC\colon L^2\to L^2$ and $\FC\colon L^1\to L^\infty$. This argument completely disregards the fact that if $f\in L^1$, then $\hat f$ is not only bounded but also continuous. However, it seems reasonable to expect that if $f\in L^p$, $1<p<2$, then $\hat f$ should be more regular than an arbitrary $L^{p'}$-function. A striking manifestation of this is a result by Tomas \cite{Tomas1975:ARestrictionTheoremForTheFourierTransform}, stating that for any fixed $p$ with $1\le p<2(d+1)/(d+3)$,
\begin{align*}
\int_{|\xi|=1}|\hat f(\xi)|^2\,d \sigma(\xi)\lesssim \int_{\RB^d}|f(x)|^p\,d  x,\quad f\in L^p.
\end{align*}
Here $\,d \sigma$ signifies integration with respect to $(d-1)$-dimensional surface measure. We refer to \cite{Stein1993:HarmonicAnalysis:Real-VariableMethodsOrthogonalityAndOscillatoryIntegrals} for a background on Fourier restriction theorems, and to \cite{Bourgain-Guth2011:BoundsOnOscillatoryIntegralOperatorsBasedOnMultilinearEstimates} for a more recent development.

The proof of Tomas' result is based on a dyadic decomposition of frequencies, and averaging the Hausdorff--Young inequality over different frequency scales. Similar arguments appear also in Hörmander's treatment of the (closely related) Bochner--Riesz problem \cite{Hormander1973:OscillatoryIntegralsAndMultipliersOnFLp}. However, the following result does not appear to be recorded.
\begin{theorem}\label{theorem:LebesgueToBesov}
	If $1\le p\le 2$, then $\FC\colon L^p\to \dot B^{p',p}_0\cap B^{p',p}_0$ is bounded.
\end{theorem}
In the above theorem, $\dot B^{p',p}_0$ and $B^{p',p}_0$ respectively denote homogeneous and non-homogeneous Besov-spaces. Theorem~\ref{theorem:LebesgueToBesov} is significantly stronger than the Hausdorff--Young theorem. Consider for example the embeddings
\[
B^{p',p}_0\subsetneq F^{p',p}_0\subsetneq F^{p',2}_0=L^{p'},
\] 
valid for $1<p<2$, e.g. \cite[Proposition~2.3.2.2]{Triebel1983:TheoryOfFunctionSpacesI}. The inclusions are strict by \cite[Theorem~2.3.9]{Triebel1983:TheoryOfFunctionSpacesI}.

A way to think about Theorem~\ref{theorem:LebesgueToBesov} is as follows: It is known that if $\MC(\RB_+)$ denotes the space of finite complex measures on $\RB_+$, then $\LC \MC(\RB_+) \subset B^{\infty,1}_0(\RB)$, e.g. \cite[p.~257]{Vitse2005:ABesovClassFunctionalCalculusForBoundedHolomorphicSemigroups}. The case $p=1$ of Theorem~\ref{theorem:LebesgueToBesov} is but a simple variation of this result, while the case $p=2$ is the Plancherel theorem. Once again, the intermediate cases can be obtained by complex interpolation, e.g. \cite[Theorem~6.4.5]{Bergh-Lofstrom1976:InterpolationSpacesAnIntroduction}. Since $B^{\infty,1}_0$ is a space of continuous functions, the interpolation argument now reflects the fact that $\FC L^1$ consists of continuous functions. This may explain why arguments similar to the proof of Theorem~\ref{theorem:LebesgueToBesov} also appear in the literature on restriction theorems.

A key tool for us is the method of complex interpolation. The basic idea is that if $T\colon L^{p_0}+L^{p_1}\to L^{q_0}+L^{q_1}$ is a linear map, and $T\colon L^{p_j}\to L^{q_j}$ is bounded for $j\in\{0,1\}$, then $T\colon L^p\to L^q$ is bounded whenever $\left(1/p,1/q\right)$ belongs to the straight line segment connecting the points $\left(1/p_0,1/q_0\right)$ and $\left(1/p_1,1/q_1\right)$ in $\RB^2$. Now note that if $p=q=1$, then \eqref{eq:CarlesonCondition} just means that $\mu$ is a finite measure, while for any $f\in L^1(\RB_+)$, $\LC f$ is a bounded function on $\CB_+$. Hence, \eqref{eq:CarlesonCondition} implies that $\LC\colon L^1(\RB_+)\to L^1(\CB_+,d \mu)$. Based on Figure~\ref{figure:KnownAndNewResults}, it seems difficult not to imagine the existence of an interpolation result which allows for the hypothesis of sectoriality to be relaxed also in case (III). We do not resolve this problem, but we do note by means of an example that Stein--Weiss interpolation, in the sense of Theorem~\ref{theorem:Stein--Weiss} below, applied in a simple but quite general way, is not sufficient for this purpose.

The remainder of this paper is organized as follows: In Section~\ref{sec:PreliminariesAndNotatiion}, we recall some basic results and standard notation. We prove the Theorems~\ref{theorem:HardyLittlewoodInequalityNonRadial} through \ref{theorem:LebesgueToBesov} in Section~\ref{sec:FourierResults}, and apply these to the Laplace transform, in order to prove the Theorems~\ref{theorem:LaplaceCarlesonEmbeddingTheorem} through \ref{theorem:WeightedLaplaceCarlesonEmbedding}, in Section~\ref{sec:LaplaceResults}. In Section~\ref{sec:NonResult}, we give an example related to the above case (III) for non-sectorial measures.

\section{Preliminaries and notation}\label{sec:PreliminariesAndNotatiion}

Given two parametrized sets of non-negative numbers $\{A_i\}_{i\in I}$ and $\{B_i\}_{i\in I}$, we write $A_i\lesssim B_i$, $i\in I$, to indicate the existence of a constant $C>0$ such that $i\in I\implies A_i\le CB_i$. The index set $I$ is often implicit from context, in which case we allow ourselves to suppress it in our notation. If $A_i\lesssim B_i$ and $B_i\lesssim A_i$, then we write $A_i\approx B_i$.

Given an analytic function $F\colon \CB_+\to \CB$, and $y>0$, define $F_y\colon \RB \to \CB$ by $F_y(x)=F(x+iy)$. The Hardy space $H^p(\CB_+)$ is the space of analytic function $F\colon \CB_+\to \CB$ such that
\begin{align*}
\|F\|_{H^p(\CB_+)}:=\sup_{y>0}\|F_y\|_{L^p(\RB)}<\infty.
\end{align*}
If $F\in H^p(\CB_+)$, then the limit $bF(x)=\lim_{y\to 0^+}F_y(x)$ exists for Lebesgue a.e. $x\in\RB$. Moreover, $F_y\to bF$ in $L^p(\RB)$, and we may recover $F$ from $bF$ via the Poisson extension operator;
\begin{align*}
F(x+iy)=\left(P_y\ast bF\right)(x) :=
\frac{1}{\pi}\int_{t\in\RB}\frac{y}{(x-t)^2+y^2}bF(t)\,d  t.
\end{align*}
The correspondence between $F$ and $bF$ characterizes $H^p(\CB_+)$ as the subspace of $L^p(\RB)$ consisting of functions whose Poisson extensions to $\CB_+$ are analytic. We refer to \cite[Chapter~II, Section~3]{Garnett2007:BoundedAnalyticFunctions}.

In the introduction, we needed the following result on Hardy spaces:
\begin{theorem}\label{theorem:CarlesonDuren}
	Let $1<p\le q<\infty$, and $\mu$ be a positive Borel measure on $\CB_+$. Then $H^p(\CB_+)\subset L^q(\CB_+,d \mu)$ in the sense of a continuous embedding if and only if
	\begin{align*}
	\mu(Q_I)\lesssim|I|^{q/p}\quad\textnormal{for all intervals}\quad I\subset\RB.
	\end{align*}
\end{theorem}
In the case $p=q$, this is the celebrated Carleson embedding theorem \cite{Carleson1962:InterpolationsByBoundedAnalyticFunctionsAndTheCoronaProblem}, while the general case is due to Duren \cite{Duren1969:ExtensionOfATheoremOfCarleson}. 

Given $\alpha>-1$, the standard weighted Bergman space $A^p_\alpha (\CB_+)$ is the space of analytic function $F\colon \CB_+\to \CB$ such that
\begin{align*}
\|F\|_{A^p_\alpha(\CB_+)}^p:=\int_{y=0}^\infty \int_{x\in\RB}|F(x+iy)|^py^\alpha \,d  x\,d  y<\infty.
\end{align*}
A Bergman space analogue of Theorem~\ref{theorem:CarlesonDuren} is easily derived from \cite[Theorem~2.1]{Jacob-Partington-Pott2013:OnLaplace-CarlesonEmbeddingTheorems}, or by the method outlined in \cite{Luecking1983:ATechniqueForCharacterizingCarlesonMeasuresOnBergmanSpaces}:
\begin{theorem}\label{theorem:CarlesonBergman}
	Let $1<p<\infty$, $\alpha>-1$, and $\mu$ be a positive Borel measure on $\CB_+$. Then $A^p_\alpha(\CB_+)\subset L^p(\CB_+,d \mu)$ in the sense of a continuous embedding if and only if
	\begin{align*}
	\mu(Q_I)\lesssim|I|^{2+\alpha}\quad\textnormal{for all intervals}\quad I\subset\RB.
	\end{align*}
\end{theorem}

We let $\SC$ denote the Schwartz class of functions on $\RB^d$, and $\SC'$ its topological dual. The Fourier transform $\FC\colon f\mapsto \hat f$, $f\in\SC$, is defined according to the convention
\begin{align*}
\hat f(\xi)=\int_{\RB^d}f(x)e^{-2\pi i x\cdot \xi}\,d  x,\quad \xi\in\RB^d,
\end{align*}
and extended to $\SC'$ by the relation $\langle \hat f,g\rangle = \langle f,\hat g\rangle$. 

We note that $\LC f(x+iy)=\FC^{-1}(e^{-2\pi y \cdot}f)(x)$. In particular, if $\Phi\in\SC(\RB)$ satisfies $\hat\Phi(\xi)=e^{-2\pi \xi}$ for $\xi\ge 0$, and $\Phi_y$ denotes the $L^1(\RB)$-normalized dilation $x\mapsto \frac{1}{y}\Phi(\frac{x}{y})$, then $\LC f(x+iy)=\left(\Phi_y\ast \check f\right)(x)$, where $\check f=\FC^{-1} f$. If $f\in L^p(\RB_+)$, $1\le p\le 2$, so that $\check f\in L^{p'}(\RB)$, then we may replace $\Phi$ with the Poisson kernel $P\colon x\mapsto \frac{1}{\pi}\frac{1}{1+x^2}$, since $P\in L^1(\RB)$ and $\hat P (\xi)=e^{-2\pi |\xi|}$. Consequently, $\LC\colon L^p(\RB_+)\to H^{p'}(\CB_+)$ is bounded by the Hausdorff--Young theorem.

The subspace $\SC_0\subset\SC$ is defined by the condition that $\int f(x)x^\alpha\,d  x=0$ for all multi-indices $\alpha$, or equivalently that any derivative of $\hat f$ vanishes at the origin. Its dual coincides with $\SC'/\PC$, where $\PC$ denotes the space of polynomials. For a discussion on $\SC_0$ and its dual, we refer to \cite[Chapter~5]{Triebel1983:TheoryOfFunctionSpacesI}. Said monograph is also a standard reference for the following material on Besov- and Triebel--Lizorkin-spaces.

The Bessel potential $I_\alpha\colon f\mapsto \FC^{-1}\left(\left(1+|\cdot|^2\right)^{\alpha/2}\hat f\right)$ is a homeomorphism on $\SC'$, whenever $\alpha\in\RB$. Similarly, the Riesz potential $\dot I_\alpha\colon  f \mapsto \FC^{-1}\left(|\cdot|^\alpha \hat f\right)$ is a homeomorphism on $\SC'/\PC$.

Let $\varphi\in\SC$. Assume that $\hat\varphi$ is radially decreasing, $\hat \varphi(\xi) = 1$ for $|\xi|\le 1$, and $\hat \varphi(\xi)=0$ for $|\xi|\ge 2$. Define a sequence $(\varphi_k)_{k\in\ZB}$, by $\hat \varphi_0(\xi)=\hat \varphi(\xi/2)-\hat \varphi(\xi)$, and $\hat \varphi_{k}(\xi)=\hat \varphi_0(2^{-k}\xi)$ for $k\ne 0$. It then holds that $\hat \varphi+\sum_{k=0}^\infty\hat\varphi_k\equiv 1$ on $\RB^d$, and $\sum_{k=-\infty}^\infty \hat \varphi_k\equiv 1$ on $\RB^d\setminus\{0\}$. For $1< p,q<\infty$, $s\in\RB$, and $f\in\SC'$, let
\begin{align*}
\|f\|_{F^{p,q}_s}=\left\|\left(\left|\varphi\ast f\right|^q+\sum_{k=0}^\infty \left|2^{ks}(\varphi_k\ast f)\right|^q\right)^{1/q}\right\|_{L^p}.
\end{align*}
It can be shown that $\|f\|_{F^{p,q}_s}$ is independent of the choice of $\varphi$, in the sense of equivalent norms. Hence, we may define the non-homogeneous Triebel--Lizorkin space $F^{p,q}_s$ as
\[
F^{p,q}_s:=\left\{f\in\SC' \mid  \|f\|_{F^{p,q}_s} <\infty \right\}.
\]
This is a Banach space, and the Bessel potential acts as a shift operator on the smoothness index $s$: If $s,\alpha\in\RB$, then $I_\alpha\colon F^{p,q}_s\to F^{p,q}_{s-\alpha}$ is a bounded isomorphism of Banach spaces. In particular, $\|I_\alpha f \|_{F^{p,q}_{s-\alpha}}\approx \|f\|_{F^{p,q}_{s}}$.

Similarly, let
\begin{align*}
\|f\|_{\dot F^{p,q}_s}=\left\|\left(\sum_{k=-\infty}^\infty \left|2^{ks}(\varphi_k\ast f)\right|^q\right)^{1/q}\right\|_{L^p}.
\end{align*}
Note that $\|f\|_{\dot F^{p,q}_s}=0$ if and only if $\supp \hat f\subseteq \{0\}$, i.e. if and only if $f\in\PC$. The homogeneous Triebel--Lizorkin space $\dot F^{p,q}_s$ is defined as
\[
\dot F^{p,q}_s:=\left\{[f]\in\SC'/\PC \mid  \|f\|_{\dot F^{p,q}_s} <\infty \right\}.
\]
This is also a Banach space, with the Riesz potential acting as a shift of smoothness: If $s,\alpha\in\RB$, then $\dot I_\alpha\colon \dot F^{p,q}_s\to \dot F^{p,q}_{s-\alpha}$ is a bounded isomorphism of Banach spaces. In particular, $\|\dot I_\alpha f \|_{\dot F^{p,q}_{s-\alpha}}\approx \|f\|_{\dot F^{p,q}_{s}}$.

Let $\psi_k=\varphi_{k-1}+\varphi_{k}+\varphi_{k+1}$. Then $\hat \psi_k\equiv 1$ on the support of $\hat\varphi_{k}$.  For $f\in\SC'$, we consider the formal series
\[
f_0=\sum_{k=-\infty}^\infty \varphi_k\ast f=\sum_{k=-\infty}^\infty \psi_k\ast \varphi_k\ast f.
\]
If this series converges in $\SC'$, then we call $f_0$ the canonical representative of $[f]\in\SC'/\PC$. It is an exercise to show that $\sum_{k=0}^{\infty}\varphi_{k}\ast f$ always converges in $\SC'$. As for the other half of the series, it is trivial that
\[
\|\varphi_{k}\ast f\|_{L^p}\le 2^{-sk}\|f\|_{\dot F^{p,q}_s},\quad f\in\SC'.
\]
Hence, if $s<0$, and $f\in {\dot F^{p,q}_s}$, then the series $\sum_{k=-\infty}^{-1}\varphi_{k}\ast f$ converges in $L^p$. For $s=0$, we first use Young's inequality to obtain that
\[
\|\psi_k\ast \varphi_{k}\ast f\|_{L^r}\le \|\psi_k\|_{L^q}\|\varphi_{k}\ast f\|_{L^p}=2^{kd/q'}\|\psi_0\|_{L^q}\|\varphi_{k}\ast f\|_{L^p},
\]
whenever $\frac{1}{p}+\frac{1}{q}=1+\frac{1}{r}$. In particular, $\sum_{k=-\infty}^{-1}\varphi_{k}\ast f$ converges in $L^r$ for any $r>p$. We conclude that if $s\le 0$, then any $f\in \dot F^{p,q}_s$ has a canonical representative $f_0$. If $s<0$, then it is easy to see that $\varphi\ast f_0\in L^p$, and that $f_0\in F^{p,q}_s$. A somewhat deeper fact is the Littlewood--Paley theorem: With the above identification, $L^p=\dot F^{p,2}_0=F^{p,2}_0$. 

We define $W^p_s$, $1<p<\infty$, as the space of $f\in\SC'$, such that $I_s f\in L^p$, i.e. $W^p_s=F^{p,2}_s$. Similarly, $\dot W^p_s=\dot F^{p,2}_s$.

The definition of the Besov spaces $B^{p,q}_s$ and $\dot B^{p,q}_s$ is similar to that of $F^{p,q}_s$ and $\dot F^{p,q}_s$; we only interchange the $L^p$- and $\ell^q$-norms. In other words, the norms are given by
\begin{align*}
\|f\|_{B^{p,q}_s}=\left(\|\varphi\ast f\|_{L^p}^q+\sum_{k=0}^\infty 2^{ksq}\|\varphi_k\ast f\|_{L^p}^q\right)^{1/q},
\end{align*}
and 
\begin{align*}
\|f\|_{\dot B^{p,q}_s}=\left(\sum_{k=-\infty}^\infty 2^{ksq}\|\varphi_k\ast f\|_{L^p}^q\right)^{1/q},
\end{align*}
and the spaces $B^{p,q}_s\subset S'$ and $\dot B^{p,q}_s\subset \SC'/\PC$ are defined by imposing finiteness of the respective norm. Here we also allow for the endpoints $p,q\in\{1,\infty\}$. If $1<p<\infty$, then $F^{p,p}_s=B^{p,p}_s$, and $\dot F^{p,p}_s=\dot B^{p,p}_s$.

Since the spaces $\ell^q(\ZB)$ increase with $q$, the same is true for the spaces $F^{p,q}_s$, $\dot F^{p,q}_s$, $B^{p,q}_s$, and $\dot B^{p,q}_s$. A more sophisticated embedding result is given by \cite[Theorem~2.7.1]{Triebel1983:TheoryOfFunctionSpacesI}:
\begin{theorem}\label{theorem:TriebelLizorkinEmbeddingTheorem}
	If $1<p_0,q_0,p_1,q_1<\infty$, $s_1<s_0$, and $s_0-\frac{d}{p_0}=s_1-\frac{d}{p_1}$, then
	\begin{align*}
	F^{p_0,q_0}_{s_0}\subset F^{p_1,q_1}_{s_1}\qquad\textnormal{and}\qquad \dot F^{p_0,q_0}_{s_0}\subset \dot F^{p_1,q_1}_{s_1}.
	\end{align*}
\end{theorem}
We will frequently exploit that if $s_0-\frac{d}{p_0}=s_1-\frac{d}{p_1}$, then $s_1<s_0$ if and only if $p_0<p_1$.

We need the following instance of the so-called Stein--Weiss interpolation theorem,  e.g. \cite[Corollary~5.5.4]{Bergh-Lofstrom1976:InterpolationSpacesAnIntroduction}.
\begin{theorem}\label{theorem:Stein--Weiss}
	Consider two measure spaces $(X,\mu)$ and $(Y,\nu)$. For $j\in\{0,1\}$, let $v_j\colon X\to[0,\infty]$ and $w_j\colon Y\to[0,\infty]$ be measurable functions, and $1\le p_j,q_j<\infty$. Assume further that 
	\begin{align*}
	T\colon L^{p_0}(X,v_0\,d \mu)+L^{p_1}(X,v_1\,d \mu)\to L^{q_0}(Y,w_0\,d  \nu)+L^{q_1}(Y,w_1\,d  \nu)
	\end{align*}
	is a linear map, and that
	\begin{align*}
	T\colon L^{p_j}(X,v_j\,d \mu)\to L^{q_j}(Y,w_j\,d  \nu)
	\end{align*}
	is bounded for $j\in\{0,1\}$. If
	\begin{align*}
	\frac{1}{p}=\frac{1-\theta}{p_0}+\frac{\theta}{p_1},\qquad 
	\frac{1}{q}=\frac{1-\theta}{q_0}+\frac{\theta}{q_1},
	\end{align*}
	and
	\begin{align*}
	v=v_0^{(1-\theta)\frac{p}{p_0}}v_1^{\theta\frac{p}{p_1}},\qquad
	w=w_0^{(1-\theta)\frac{q}{q_0}}w_1^{\theta\frac{q}{q_1}},
	\end{align*}
	for some $\theta\in (0,1)$,	then
	\begin{align*}
	T\colon L^{p}(X,v\,d \mu)\to L^{q}(Y,w\,d  \nu)
	\end{align*}
	is bounded.
\end{theorem}

\section{Proofs of Theorems~\ref{theorem:HardyLittlewoodInequalityNonRadial} through \ref{theorem:LebesgueToBesov}}\label{sec:FourierResults}

Given $x=(x_1,\ldots,x_d)\in\RB^d$, we write $\Pi_x=\prod_{k=1}^d|x_k|$.

\begin{lemma}\label{lemma:ConvolutionInequality}
	Let $p>1$. If $f\in L^1$, then
	\begin{align*}
	\int_{\RB^d}|(f\ast f)(x)|^p \Pi_x{}^{p-2}\,d  x
	\lesssim
	\int_{\RB^d}|f(x)|^{2p} \Pi_x{}^{2p-2}\,d  x.
	\end{align*}
\end{lemma}
\begin{proof}
	For a choice of $\alpha\in\RB$, it holds that
	\begin{align}\label{eq:ExponentInequalities}
	\frac{p-1}{2p}<\frac{2p-2}{2p+1}<\alpha<\frac{p-1}{p}<1.
	\end{align}
	Multiplying the corresponding integrand by $1=\Pi_{x-y}{}^\alpha \Pi_{y}{}^\alpha\Pi_{x-y}{}^{-\alpha}\Pi_{y}{}^{-\alpha}$, the convolution
	\begin{align*}
	(f\ast f)(x)
	&=
	\int_{y\in\RB^d}f(x-y)\Pi_{x-y}{}^\alpha f(y)\Pi_{y}{}^\alpha\Pi_{x-y}{}^{-\alpha}\Pi_{y}{}^{-\alpha}\,d  y.
	\end{align*}
	By Hölder's inequality,
	\begin{align*}
	|(f\ast f)(x)|^p
	\le {}&
	\int_{y\in\RB^d}|f(x-y)|^p\Pi_{x-y}{}^{\alpha p}|f(y)|^p\Pi_{y}{}^{\alpha p}\,d  y
	\\
	&\times
	\left(\int_{y\in\RB^d}\Pi_{x-y}{}^{-\alpha p'}\Pi_{y}{}^{-\alpha p'}\,d  y\right)^{p-1}.
	\end{align*}
	By \eqref{eq:ExponentInequalities}, $\frac{1}{2}<\alpha p'<1$, and a change of variables yields
	\begin{align*}
	\int_{y_k\in\RB}\frac{1}{|x_k-y_k|^{\alpha p'}|y_k|^{\alpha p'}}\,d  y_k
	&=
	\frac{1}{|x_k|^{\alpha p'}}\int_{y_k\in\RB}\frac{1}{|1-\frac{y_k}{x_k}|^{\alpha p'}|y_k|^{\alpha p'}}\,d  y_k
	\\
	&=
	|x_k|^{1-2\alpha p'}\int_{y_k\in\RB}\frac{1}{|1-y_k|^{\alpha p'}|y_k|^{\alpha p'}}\,d  y_k.
	\end{align*}
	(In the sense of extended real numbers, the above equalities are valid even for $x_k=0$.) Therefore,
	\begin{align}\label{eq:HomogeneousIntegralIdentity}
	\left(\int_{y\in\RB^d}\Pi_{x-y}{}^{-\alpha p'}\Pi_{y}{}^{-\alpha p'}\,d  y\right)^{p-1}
	=
	c\Pi_{x}{}^{p-1-2\alpha p},
	\end{align}
	for some finite $c>0$, and
	\begin{align*}
	|(f\ast f)(x)|^p\Pi_{x}{}^{p-2}
	\lesssim
	\Pi_{x}{}^{2p-3-2\alpha p}\int_{y\in\RB^d}|f(x-y)|^p\Pi_{x-y}{}^{\alpha p}|f(y)|^p\Pi_{y}{}^{\alpha p}\,d  y.
	\end{align*}
	Integration with respect to $x\in\RB^d$, and another change of variables, yields
	\begin{multline*}
	\int_{x\in\RB^d}|(f\ast f)(x)|^p\Pi_{x}{}^{p-2}\,d  x
	\\
	\lesssim
	\iint_{x,y\in\RB^d}|f(x)|^p\Pi_{x}{}^{\alpha p}|f(y)|^p\Pi_{y}{}^{\alpha p}\Pi_{x+y}{}^{2p-3-2\alpha p}\,d  y\,d  x.
	\end{multline*}
	Multiplying this integrand by $1=\Pi_{x}{}^{\alpha}\Pi_{y}{}^{\alpha}\Pi_{x}{}^{-\alpha}\Pi_{y}{}^{-\alpha}$,
	\begin{multline*}
	\iint_{x,y\in\RB^d}|f(x)|^p\Pi_{x}{}^{\alpha p}|f(y)|^p\Pi_{y}{}^{\alpha p}\Pi_{x+y}{}^{2p-3-2\alpha p}\,d  y\,d  x
	\\
	=
	\iint_{x,y\in\RB^d}|f(x)|^p\Pi_{x}{}^{\alpha p+\alpha}|f(y)|^p\Pi_{y}{}^{\alpha p+\alpha}\Pi_{x+y}{}^{2p-3-2\alpha p}\Pi_{x}{}^{-\alpha}\Pi_{y}{}^{-\alpha}\,d  y\,d  x.
	\end{multline*}
	To the above right-hand side, apply the elementary inequality $ab\le \frac{a^2+b^2}{2}$, with 
	\[
	a=|f(x)|^p\Pi_{x}{}^{\alpha p + \alpha},\quad \textnormal{and}\quad b=|f(y)|^p\Pi_{y}{}^{\alpha p + \alpha},
	\]
	and use that the two resulting integrals are are equal, to obtain that
	\begin{multline*}
	\iint_{x,y\in\RB^d}|f(x)|^p\Pi_{x}{}^{\alpha p}|f(y)|^p\Pi_{y}{}^{\alpha p}\Pi_{x+y}{}^{2p-3-2\alpha p}\,d  y\,d  x
	\\
	\le
	\iint_{x,y\in\RB^d}|f(x)|^{2p}\Pi_{x}{}^{2\alpha p+\alpha }\Pi_{x+y}{}^{2p-3-2\alpha p}\Pi_{y}{}^{-\alpha}\,d  y\,d  x.
	\end{multline*}
	By \eqref{eq:ExponentInequalities}, $\alpha<1$, $2\alpha p+3-2p<1$, and $\alpha+2\alpha p+3-2p>1$. By an argument similar to the one leading up to \eqref{eq:HomogeneousIntegralIdentity},
	\begin{align*}
	\int_{y\in\RB^d}\Pi_{x+y}{}^{2p-3-2\alpha p}\Pi_{y}{}^{-\alpha}\,d  y
	=
	c\Pi_{x}{}^{2p-2-2\alpha p-\alpha},
	\end{align*}
	for some $c\in\RB$. This completes the proof.
\end{proof}

\begin{proof}[Proof of Theorem~\ref{theorem:HardyLittlewoodInequalityNonRadial}]
	The statement is that if $p\ge 2$, and $f\in L^1$, then
	\begin{align*}
	\int_{\RB^d}|\hat f(\xi)|^p\,d  \xi
	\lesssim
	\int_{\RB^d}|f(x)|^{p}\Pi_{x}{}^{p-2}\,d  x.
	\end{align*}
	We will prove the statement for $p=2^N$, $N\in\ZB_{\ge 1}$. The general result follows by Stein--Weiss interpolation, Theorem~\ref{theorem:Stein--Weiss}.
	
	Let $f_0=f$, and $f_{N}=f_{N-1}\ast f_{N-1}$, $N\in\ZB_{\ge 1}$, so that $\hat f_N=\hat f^{2^N}$. By the Plancherel theorem,
	\begin{align*}
	\int_{\RB^d}|\hat f(\xi)|^{2^N}\,d \xi
	=
	\int_{\RB^d}|\hat f_{N-1}(\xi)|^{2}\,d \xi
	=
	\int_{\RB^d}|f_{N-1}(x)|^{2}\,d  x.
	\end{align*}
	Combining Lemma~\ref{lemma:ConvolutionInequality} with an induction argument,
	\begin{align*}
	\int_{\RB^d}|\hat f(\xi)|^{2^N}\,d \xi
	&\lesssim
	\int_{\RB^d}|f_{N-k}(x)|^{2^k}\Pi_{x}{}^{2^k-2}\,d  x,
	\end{align*}
	for $k=1,\ldots,N$. In particular, the desired inequality holds for $p=2^{N}$.
\end{proof}


\begin{proof}[Proof of Theorem~\ref{theorem:LebesgueToHomogneneousSobolev}]
	With $(\varphi_{k})_{k\in\ZB}$ as in the definition of $\dot F^{p,q}_s$, and $f\in L^p$, let $f_N=\sum_{k=-N}^N\hat\varphi_k f$, and $g_N(x)=|x|^{s}f_N(x)$, where $s=d\left(\frac{2}{p}-1\right)$. Since $g_N\in L^1$, Theorem~\ref{theorem:HardyLittlewoodInequalityRadial} implies that
\[
\int_{\RB^d}|\hat g_N(\xi)|^p\,d \xi\lesssim \int_{\RB^d}|f_N(x)|^p\,d  x.
\]
But $\hat g_N=\dot I_s\hat f_N$, so we may use the Littlewood--Paley theorem, and the lifting property of $\dot I_s$, to obtain that
\begin{align*}
\|\hat f_N\|_{\dot F^{p,2}_s}=\left\|\left(\sum_{k=-\infty}^\infty \left|2^{ks}(\varphi_k\ast \hat f_N)\right|^q\right)^{1/q}\right\|_{L^p}\lesssim \|f_N\|_{L^p}\le \|f\|_{L^p}.
\end{align*}
Since $f_N\to f$ in $L^p$, $\hat f_N\to \hat f$ in $\SC'$. Moreover, $\lim_{N\to\infty}(\varphi_k\ast \hat f_N)(x)=(\varphi_k\ast \hat f)(x)$ for every $k$ and $x$. A standard application of Fatou's lemma implies that $\hat f\in \dot W^p_s$.
\end{proof}

\begin{proof}[Proof of Theorem~\ref{theorem:OptimalityOfResults}]
	First, note that $\FC\colon L^p\to \dot F^{r,q}_s$ is bounded if and only if $\FC^{-1}\colon L^p\to \dot F^{r,q}_s$ is bounded.
	
	With $(\varphi_{k})_{k\in\ZB}$ as in the definition of $\dot F^{p,q}_s$,
	\begin{align*}
	\|\varphi_n\|_{\dot F^{r,q}_s}
	=
	\left\|\left(\sum_{k=-\infty}^\infty \left|2^{ks}(\varphi_k\ast \varphi_n)\right|^q\right)^{1/q}\right\|_{L^r}
	\ge
	\left\|2^{ns}(\varphi_n\ast \varphi_n)\right\|_{L^r}.
	\end{align*}
	By some simple changes of variables, $\left(\varphi_n\ast\varphi_n\right)(x)=2^{nd}\left(\varphi_0\ast\varphi_0\right)(2^nx)$, and
	\begin{align*}
	\left\|(\varphi_n\ast \varphi_n)\right\|_{L^r}=2^{nd/r'}\left\|(\varphi_0\ast \varphi_0)\right\|_{L^r}.
	\end{align*}
	If $\FC^{-1}\colon L^p\to \dot F^{r,q}_s$ is bounded, then
	\begin{align*}
	2^{n(s+d/r')}\lesssim \|\varphi_n\|_{\dot F^{r,q}_s} \lesssim \|\hat \varphi_n\|_{L^p}
	=
	2^{nd/p}\|\hat \varphi_0\|_{L^p}.
	\end{align*}
	Such an inequality is only possible if $s=d/p-d/r'$.
	
	In order to obtain a contradiction, assume now that $\FC^{-1}\colon L^p\to \dot F^{r,q}_{d/p-d/r'}$ is bounded, and that $r<p$. By Theorem~\ref{theorem:TriebelLizorkinEmbeddingTheorem}, $\dot F^{r,q}_{d/p-d/r'}\subset \dot F^{\tilde r,\tilde r}_{d/p-d/\tilde r'}$ whenever $r<\tilde r$. It therefore suffices to obtain a contradiction in the case where $q=r<p$. Given a sequence $(\alpha_n)_{n\in\ZB }\in\ell^p(\ZB)$, define $f=\sum_{n\in\ZB }\alpha_n2^{-2nd/p}\varphi_{2n}$. Using that the functions $\left(\hat\varphi_{2n}\right)_{n\in\ZB }$ have pairwise disjoint supports, $\|\hat f\|_{L^p}^p=\left\|(\hat \varphi_0)\right\|_{L^p}^p\sum_{n\in\ZB }|\alpha_n|^p$, and $\left(\varphi_{2k}\ast f\right) = \alpha_k2^{-2kd/p}\left(\varphi_{2k}\ast\varphi_{2k}\right)$. Using the assumption that $\FC^{-1}\colon L^p\to \dot F^{r,r}_{d/p-d/r'}$ is bounded,
	\begin{align*}
	\left(\sum_{n=-\infty}^\infty|\alpha_n|^p\right)^{1/p}
	&\gtrsim
	\left(\int_{x\in \RB^d}\sum_{k=-\infty}^\infty \left|2^{k(d/p-d/r')}(\varphi_k\ast f)(x)\right|^r\,d  x\right)^{1/r}
	\\
	&\ge
	\left(\int_{x\in \RB^d}\sum_{k=-\infty}^\infty \left|2^{2k(d/p-d/r')}(\varphi_{2k}\ast f)(x)\right|^r\,d  x\right)^{1/r}
	\\
	&=
	\left(\int_{x\in \RB^d}\sum_{k=-\infty}^\infty 2^{2kd(1-r)}|\alpha_k|^r\left|(\varphi_{2k}\ast \varphi_{2k})(x)\right|^r\,d  x\right)^{1/r}
	\\
	&=
	\left\|(\varphi_0\ast \varphi_0)\right\|_{L^r}\left(\sum_{k=-\infty}^\infty |\alpha_k|^r\right)^{1/r}.
	\end{align*}
	Hence, we have derived that $\ell^p\subset\ell^r$, which is obviously false for $r<p$.
\end{proof}

\begin{proof}[Proof of Theorem~\ref{theorem:LebesgueToBesov}]
	As in the proof of Theorem~\ref{theorem:OptimalityOfResults}, it is more convenient to show that $\FC^{-1}\colon L^p \to \dot B^{p',p}_0\cap B^{p',p}_0$. Since $\FC (\varphi_k\ast f)=\hat \varphi_k\hat f$, the Hausdorff--Young theorem implies that
	\begin{align*}
	\|f\|_{\dot B^{p',p}_0}^p
	=
	\sum_{k=-\infty}^\infty \|\varphi_{k}\ast f\|_{L^{p'}}^p
	\le
	\sum_{k=-\infty}^\infty \|\hat \varphi_{k}\hat  f\|_{L^{p}}^p
	=
	\int_{\RB^d}|\hat f(\xi)|^p\sum_{k=-\infty}^\infty |\hat \varphi_{k}(\xi)|^p\,d  \xi.
	\end{align*}
	Since $\sum_{k=-\infty}^\infty |\hat \varphi_{k}(\xi)|^p\le 1$ for $\xi\in \RB^d$, $\|f\|_{\dot B^{p',p}_0}\le \|\hat f\|_{L^p}$. To obtain control of the non-homogeneous norm as well, it is enough to note that by Young's inequality
	\begin{align*}
	\|\varphi\ast f\|_{L^{p'}}
	\le 
	\|\varphi\|_{L^{1}}\|f\|_{L^{p'}}
	\le 
	\|\varphi\|_{L^{1}}\|\hat f\|_{L^{p}}.
	\end{align*}
\end{proof}

\section{Proof of Theorems~\ref{theorem:LaplaceCarlesonEmbeddingTheorem} through \ref{theorem:WeightedLaplaceCarlesonEmbedding}}\label{sec:LaplaceResults}

As was mentioned in Section~\ref{sec:PreliminariesAndNotatiion}, the definition of $\dot F^{p,q}_s$ does not depend on the averaging kernel $\varphi$. In fact, the definition of $\dot F^{p,q}_s$ is much more flexible than we indicated. An example of a general result in this direction is \cite[Theorem~3.2]{Liang-Sawano-Ullrich-Yang-Yuan2012:NewCharacterizationsOfBesov--Triebel--Lizorkin--HausdorffSpacesIncludingCoorbitsAndWavelets}. The following special case will help us to relate Bergman spaces to the so-called Sobolev--Slobodeckij spaces $\dot F^{p,p}_s$:
\begin{theorem}\label{theorem:TriebelSpacesByLocalMeans}
	Let $1<p,q<\infty$, $s<0$, $\Phi\in\SC$, $\hat\Phi(0)\ne 0$, and $\Phi_t\colon x\mapsto \frac{1}{t^d}\Phi\left(\frac{x}{t}\right)$. If $f\in\SC'$, and $\|f\|_{\dot F^{p,q}_s}<\infty$, then its canonical representative $f_0$ satisfies
	\[
	\int_{x\in \RB^d}\left[\int_{t=0}^\infty |\Phi_t\ast f_0(x)|^q t^{-sq-1}\,d  t \right]^{p/q}\,d  x<\infty.
	\]
	Conversely, if $f_0$ satisfies the above condition, then $\|f_0\|_{\dot F^{p,q}_s}<\infty$, and $f_0$ is the canonical representative of $[f_0]\in\dot F^{p,q}_s$.
	Moreover, the above expression is comparable to $\|f_0\|_{\dot F^{p,q}_s}^p$.
\end{theorem}

\begin{proof}[Proof of Theorem~\ref{theorem:LaplaceBergmanEmbedding}]
	The statement is that $\LC \colon L^p(\RB_+)\to A^q_{q/p'-2}(\CB_+)$	is bounded, provided that $2<p\le q<\infty$. Choose $\Phi\in\SC(\RB)$ such that $\hat \Phi(t)=e^{-2\pi t}$ for $t\ge 0$, and let $f\in L^p(\RB_+)$. By Theorem~\ref{theorem:LebesgueToHomogneneousSobolev} and Theorem~\ref{theorem:TriebelLizorkinEmbeddingTheorem}, $\check f\in \dot F^{p,2}_{s_0}\subset\dot F^{q,q}_{s_1}$, where $s_0=\frac{2}{p}-1$ and $s_1=\frac{1}{q}-\frac{1}{p'}$. The dominated convergence theorem yields that $\check f$ is its own canonical representative, so Theorem~\ref{theorem:TriebelSpacesByLocalMeans} implies that 
	\[
	\int_{x\in \RB}\int_{y=0}^\infty |\Phi_y\ast \check f(x)|^q y^{-s_1q-1}\,d  y \,d  x\lesssim \|\check f\|_{\dot F^{q,q}_{s_1}}^p.
	\]
	One easily verifies that $\LC f(x+iy)=\Phi_y\ast \check f(x)$, and since $-s_1q-1=q/p'-2$, $\|\LC f\|_{A^q_{q/p'-2}(\CB_+)}^p\lesssim \|\check f\|_{\dot F^{q,q}_{s_1}}^p$.
\end{proof}

\begin{proof}[Proof of Theorem~\ref{theorem:ZBergmanEmbedding}]
	Given $2<p\le q<\infty$, we will use that $\LC \colon L^p(\RB_+)\to A^q_{q/p'-2}(\CB_+)$ to obtain the inequality
	\begin{equation}\label{eq:ZBergmanEmbedding}
	\left(\int_\DB \left|\sum_{k=0}^\infty a_kw^k\right|^q(1-|w|^2)^{q/p'-2}\,d  A(w)\right)^{1/q}\lesssim \left(\sum_{k=0}^\infty |a_k|^p\right)^{1/p}.
	\end{equation}
	
	We begin with a reduction to the case $p=q$. By Hölder's inequality,
	\[
	\left|\sum_{k=0}^\infty a_kw^k\right|
	\le
	\left(\sum_{k=0}^\infty |a_k|^p\right)^{1/p}\left(\sum_{k=0}^\infty |w|^{kp'}\right)^{1/p'}.
	\]
	Since
	\[
	\sum_{k=0}^\infty |w|^{kp'}\lesssim \frac{1}{1-|w|^{2}},\quad w\in\DB,
	\]
	\begin{align*}
	\left|\sum_{k=0}^\infty a_kw^k\right|^q(1-|w|^2)^{q/p'-2}
	&=
	\left|\sum_{k=0}^\infty a_kw^k\right|^p
	\left|\sum_{k=0}^\infty a_kw^k\right|^{q-p}(1-|w|^2)^{q/p'-2}
	\\
	&\lesssim
	\left|\sum_{k=0}^\infty a_kw^k\right|^p
	\left(\sum_{k=0}^\infty |a_k|^p\right)^{q/p-1}
	(1-|w|^2)^{p-3},
	\end{align*}
	and \eqref{eq:ZBergmanEmbedding} follows from
	\[
	\int_\DB \left|\sum_{k=0}^\infty a_kw^k\right|^p(1-|w|^2)^{p-3}\,d  A(w)\lesssim \sum_{k=0}^\infty |a_k|^p.
	\]
	
	To prove the above inequality, assume without loss of generality that the right-hand side is finite. Since $\sum_{k=0}^\infty a_kw^k=\lim_{N\to \infty}\sum_{k=0}^N a_kw^k$ for $w\in\DB$, Fatou's lemma allows us to only consider sequences with finitely many non-zero elements. Also, by another application of Hölder's inequality,
	\[
	\left|\sum_{k=0}^\infty a_kw^k\right|^p
	\lesssim
	\frac{1}{(1-|w|^{2})^{p-1}}\sum_{k=0}^\infty |a_k|^p.
	\]
	Since the right-hand side is dominated by $\sum_{k=0}^\infty |a_k|^p$ on any compact subset of $\DB$, it is sufficient to prove that
	\[
	\int_{r<|w|<1} \left|\sum_{k=0}^N a_kw^k\right|^p(1-|w|^2)^{p-3}\,d  A(z)\le C\sum_{k=0}^N |a_k|^p,
	\]
	for some $r$ close to $1$, and $C$ independent of $N$. By the substitution $w=e^{2\pi i z}$, $z=x+iy$,
	\begin{multline*}
	\int_{r<|w|<1} \left|\sum_{k=0}^N a_kw^k\right|^p(1-|w|^2)^{p-3}\,d  A(w)
	\\
	=
	4\pi^2\iint_{\substack{|x|<\frac{1}{2}\\ 0<y<\epsilon}} \left|\sum_{k=0}^N a_ke^{2\pi i kz}\right|^p(1-e^{-4\pi y})^{p-3}e^{-4\pi y}\,d  x\,d  y,
	\end{multline*}
	for some small $\epsilon>0$. For $z$ in the above domain of integration, $1-e^{-4\pi y}\approx y$, and $\left|\frac{e^{2\pi i z}-1}{2\pi i z}\right|\approx 1$.	If we let $F(z)=\frac{e^{2\pi i z}-1}{2\pi i z}\sum_{k=0}^N a_ke^{2\pi i kz}$,	then
	\begin{align*}
	\int_{r<|w|<1} \left|\sum_{k=0}^N a_kw^k\right|^p(1-|w|^2)^{p-3}\,d  A(w)
	&\lesssim
	\|F\|_{A^p_{p-3}(\CB_+)}^p.
	\end{align*}
	But $F=\LC f$, where $f=\sum_{k=0}^{N}a_k\1_{(k,k+1)}$. Since $\|f\|_{L^p(\RB_+)}^p=\sum_{k=0}^N|a_k|^p$, the result follows from Theorem~\ref{theorem:LaplaceBergmanEmbedding}.
\end{proof}

\begin{proof}[Proof of Theorem~\ref{theorem:WeightedLaplaceBergmanEmbedding}]
	In general, the map $f\mapsto |\cdot|^{\alpha /p}f$ takes $L^p(\RB_+,x^\alpha\,d  x)$ isometrically onto $L^p(\RB_+)$. Furthermore, we assume that $2<p\le q<\infty$, and $\alpha<p/q'-1$. By Theorem~\ref{theorem:LebesgueToHomogneneousSobolev}, the definition of the Riesz potential, and the lifting property,
	\begin{align*}
	\FC \colon L^p(\RB_+,x^\alpha \,d  x)\to \dot I_{-\alpha/p}\dot W^p_{2/p-1}(\RB)
	=
	\dot W^p_{s_0}(\RB),
	\end{align*}
	where $s_0= \frac{2+\alpha}{p}-1$. By Theorem~\ref{theorem:TriebelLizorkinEmbeddingTheorem}, $\dot W^p_{s_0}(\RB)=\dot F^{p,2}_{s_0}(\RB)\subset \dot F^{q,q}_{s_1}(\RB)$, where $s_1=\frac{1+\alpha}{p}-\frac{1}{q'}$. By our assumption on $\alpha$, $s_1<0$, so, as in the proof of Theorem~\ref{theorem:LaplaceBergmanEmbedding}, we are therefore allowed to apply Theorem~\ref{theorem:TriebelSpacesByLocalMeans} to conclude that
	\begin{align*}
	\LC \colon L^p(\RB_+,x^\alpha \,d  x)\to A^q_{-s_1q-1}(\CB_+)=A^q_{q/p'-2-\alpha q/p}(\CB_+).
	\end{align*}
\end{proof}

\begin{proof}[Proof of Theorem~\ref{theorem:WeightedLaplaceHardyEmbedding}]
	Let $p>2$. By the argument in the previous proof,
	\begin{align*}
	\FC \colon L^p(\RB_+,x^{p-2} \,d  x)\to \dot W^p_{0}(\RB)=L^p(\RB).
	\end{align*}
	The statement follows from the relation between $\FC^{-1}$ and $\LC$.
\end{proof}

\begin{proof}[Proof of Theorem~\ref{theorem:WeightedLaplaceCarlesonEmbedding}]
	Let $\lambda_I$ denote the midpoint of $Q_I$, and $f(t)=e^{-2\pi i \conj{\lambda_I} t}$, $t\ge 0$. Then
	\begin{align*}
	|\LC f(z)|=\frac{1}{2\pi}\frac{1}{\left|z-\conj{\lambda_I}\right|}\gtrsim \frac{1}{|I|},\quad z\in Q_I,
	\end{align*}
	and this bound is independent of the interval $I$. It follows that
	\begin{align*}
	\mu(Q_I)\lesssim |I|^q\int_{\CB_+}|\LC f(z)|^q\,d \mu.
	\end{align*}
	Assuming that $\LC \colon L^p(\RB_+,x^\alpha \,d  x)\to L^q(\CB_+,d \mu)$ is bounded, we obtain that $\mu(Q_I)\lesssim |I|^q\|f\|_{L^p(\RB_+,x^\alpha \,d  x)}^q$. Computing the norm of $f$, it holds that
	\begin{align*}
	\mu(Q_I)\lesssim |I|^{q/p'-\alpha q/p}.
	\end{align*}
	
	Under the assumptions $2<p\le q<\infty$, and $\alpha\le p/q'-1$, this necessary condition is also sufficient for $\LC \colon L^p(\RB_+,x^\alpha \,d  x)\to L^q(\CB_+,d \mu)$. For $\alpha < p/q'-1$, this is implied by Theorem~\ref{theorem:WeightedLaplaceBergmanEmbedding} and  Theorem~\ref{theorem:CarlesonBergman}. For $\alpha=p/q'-1$, we use instead Theorem~\ref{theorem:WeightedLaplaceHardyEmbedding} and Theorem~\ref{theorem:CarlesonDuren}.
\end{proof}

\section{A non-result for case (III)}\label{sec:NonResult}

If $\LC\colon L^{3/2}(\RB_+)\to L^{3/2}(\CB_+,d \mu)$ is bounded, then $\mu$ satisfies 
\begin{align}\label{eq:3/2--3/2}
\mu(Q_I)\lesssim|I|^{1/2}\quad\textnormal{for all intervals}\quad I\subset\RB.
\end{align}
Whether or not the converse holds is an open question, unless $\mu$ is sectorial, in which case the answer is positive.

One might attempt to use Stein--Weiss interpolation, Theorem~\ref{theorem:Stein--Weiss}, to prove that \eqref{eq:3/2--3/2} implies $\LC\colon L^{3/2}(\RB_+)\to L^{3/2}(\CB_+,d \mu)$ also for general measures. In order to do so, it appears necessary to find a measure $M$, and two functions $w_0,w_1\colon \CB_+\to [0,\infty]$, according to the following three conditions:
\begin{align}
\mu(A)=\int_{A} w_0w_1\,dM&\quad\textnormal{for all measurable sets}\quad A\subset\CB_+;\label{eq:3/2--3/2a}\\
\int_{Q_I} w_0^2\,dM\lesssim 1&\quad\textnormal{for all intervals}\quad I\subset\RB;\label{eq:3/2--3/2b}\\
\int_{Q_I} w_1^2\,dM \lesssim|I|&\quad\textnormal{for all intervals}\quad I\subset\RB.\label{eq:3/2--3/2c}
\end{align}
If this could be done, then $\LC\colon L^{1}(\RB_+)\to L^{1}(\CB_+,w_0^2\,d  M)$ and $\LC\colon L^{2}(\RB_+)\to L^{2}(\CB_+,w_1^2\,d  M)$ would both be bounded, and Theorem~\ref{theorem:Stein--Weiss} would imply that $\LC\colon L^{3/2}(\RB_+)\to L^{3/2}(\CB_+,d \mu)$ is also bounded.

The following example shows that the above strategy fails.

\begin{example}
	Consider a sum of unital point masses $\mu = \sum_{n=1}^\infty \delta_{n^2+i}$. Then \begin{align*}
	\mu (Q_I)=
	\#\left\{n\in\ZB_{\ge 1};  n^2\in I \right\} \quad\textnormal{whenever}\quad |I|\ge 1,
	\end{align*}
	and $\mu(Q_I)=0$ otherwise,	so clearly $\mu$ satisfies \eqref{eq:3/2--3/2}.
	
	Assume now that $M$, $w_0$, and $w_1$ satisfy \eqref{eq:3/2--3/2a}--\eqref{eq:3/2--3/2c}. Then $\mu$ is absolutely continuous with respect to $M$, and it is no restriction to assume that $M$ has the same support as $\mu$. Hence, we may assume that $M=\sum_{n=1}^\infty c_n\delta_{n^2+i}$ for some numbers $c_n>0$. For notational convenience, we let $w_{j,n}=w_j(n^2+i)$.
	
	By \eqref{eq:3/2--3/2a}, $w_{0,n}w_{1,n}c_n=1$ for every $n$. In particular, $w_{0,n}^2c_n=\frac{1}{w_{1,n}^2c_n}$. By \eqref{eq:3/2--3/2b}, 
	\begin{align*}
	\sum_{n=1}^\infty w_{0,n}^2c_n = \int_{\CB_+}w_0^2\,d  M<\infty,
	\end{align*}
	so $\lim_{n\to\infty}w_{0,n}^2c_n=0$, and $\lim_{n\to\infty}w_{1,n}^2c_n=\infty$. But by \eqref{eq:3/2--3/2c},
	\begin{align*}
	w_{1,n}^2c_n=\int_{Q_{[n^2,n^2+1]}}w_1^2\,d  M\lesssim 1.
	\end{align*}
	This contradiction shows that $M$, $w_0$, and $w_1$ cannot be chosen according to the conditions \eqref{eq:3/2--3/2a}--\eqref{eq:3/2--3/2c}.
\end{example}

\section*{Acknowledgements}
During the preparation of this paper, I have enjoyed interesting conversations with Jonathan Partington, Maria Carmen Reguera, Amol Sasane, Alexander Strohmaier, Jonathan Bennett, Charles Batty, Sandra Pott, and Alexandru Aleman. I would also like to thank Tino Ullrich, and the authors of \cite{Liang-Sawano-Ullrich-Yang-Yuan2012:NewCharacterizationsOfBesov--Triebel--Lizorkin--HausdorffSpacesIncludingCoorbitsAndWavelets} for their efforts in explaining \cite[Theorem~3.2]{Liang-Sawano-Ullrich-Yang-Yuan2012:NewCharacterizationsOfBesov--Triebel--Lizorkin--HausdorffSpacesIncludingCoorbitsAndWavelets}, and the anonymous referee for carefully reading this manuscript.

\end{document}